\documentclass[11pt]{amsart}
\usepackage{amsmath, amsthm, amssymb, amscd, graphicx, times, hyperref, amsfonts, epstopdf}
\newtheorem{thm}{Theorem}[section]
\newtheorem{cor}[thm]{Corollary}
\newtheorem{lem}[thm]{Lemma}

\theoremstyle{definition}

\theoremstyle{remark}
\newtheorem{rem}[thm]{Remark}

\newtheorem{note}[thm]{Note}
\newtheorem{ex}[thm]{Example}

\numberwithin{equation}{section}

\def\N{{\mathbb N}}
\def\Z{{\mathbb Z}}

\newcommand{\surj}{\twoheadrightarrow}

\def\MCG{{\mathcal MCG}}

\begin{document}
\title{Separating Pants Decompositions in the Pants Complex}
\author{Harold Sultan}
\address{Department  of Mathematics\\Columbia University\\
New York\\NY 10027}
\email{HSultan@math.columbia.edu}
\date{\today}
%\commby{}
% ----------------------------------------------------------------
\begin{abstract} 
We study the topological types of pants decompositions of a surface by associating to any pants decomposition $P \in \mathcal{P}(S_{g,n})$, in a natural way its \emph{pants decomposition graph}, $\Gamma(P).$  This perspective provides a convenient way to analyze the maximum distance in the pants complex of any pants decomposition to a pants decomposition containing a non-trivial separating curve for all surfaces of finite type, $S_{g,n}.$  In the main theorem we provide an asymptotically sharp approximation of this non-trivial distance in terms of the topology of the surface.  In particular, for closed surfaces of genus $g$ we show the maximum distance in the pants complex of any pants decomposition to a pants decomposition containing a separating curve grows asymptotically like the function $\log(g).$ \end{abstract}
\maketitle

\tableofcontents

\section{Introduction} \label{sec:back}
The large scale geometry of Teichm\"uller space has been an object of interest in recent years, especially within the circles of ideas surrounding Thurston's Ending Lamination Conjecture.  In this context, the pants complex, $\mathcal{P}(S),$ associated to a hyperbolic surface, $S,$ becomes relevant, as by a theorem of Jeff Brock in \cite{brock}, the pants complex is quasi-isometric to the Teichm\"uller space of a surface equipped with the Weil-Petersson metric, $(\mathcal{T}(S),d_{WP}).$  Accordingly, in order to study large scale geometric properties of Teichm\"uller space with the Weil-Petersson metric, it suffices to study the pants complex of a surface.  For instance, significant recent results of Brock-Farb \cite{brockfarb}, Behrstock \cite{behrstock}, Behrstock-Minsky \cite{behrstockminsky}, and Brock-Masur \cite{brockmasur} among others can be viewed from this perspective.

One feature of the coarse geometry of the pants complex in common to many analyses of the subject is the existence of natural quasi-isometrically embedded product regions.  These product regions, which are obstructions to $\delta$-hyperbolicity, correspond to pairs of pants decompositions of the surface containing a fixed non-trivially separating (multi)curve.  In fact, often in the course of studying the coarse geometry of the pants complex it proves advantageous to pass to the net of pants decompositions that contain a non-trivially separating curve, and hence lie in a natural quasi-isometrically embedded product region.  See for instance work of Brock-Masur in \cite{brockmasur} and Behrstock-Drutu-Mosher in \cite{bdm} in which such methods are used to prove that the pants complexes of different complexities are relatively hyperbolic or thick, respectively.  Similarly, work of Masur-Schleimer \cite{masursch}, relies on similar methods to prove the pants complex for large enough surfaces has one end.

In this paper, we study the net of pants decompositions of a surface that contain a non-trivially separating curve within the entire pants complex of a surface.  Specifically, by graph theoretic and combinatoric considerations, we determine the maximum distance in the pants complex of any pants decomposition to a pants decomposition containing a non-trivially separating curve, for all surfaces of finite type, $S_{g,n}.$  The highlight of the paper is captured by following theorem which is a slight simplification of Theorem \ref{thm:main} proven in Section \ref{sec:thm}.
\begin{thm} \label{thm:main2}
Let $S=S_{g,n}$ and set $D_{g,n} = \max_{_{P  \in \mathcal{P}(S) }} (d_{\mathcal{P}(S)} (P, \mathcal{P}_{sep}(S) )).$  Then, for any fixed number of boundary components (or punctures) $n$, $D_{g,n}$ grows asymptotically like the function $\log(g),$ that is $D_{g,n} = \Theta(\log(g)).$  On the other hand, for any fixed genus $g\geq 2, \;\forall n\geq 6g-5,$ $D_{g,n} = 2.$
\end{thm}

The non-trivial lower bounds in Theorem \ref{thm:main2} follow from an original and explicit constructive algorithm for an infinite family of high girth at most cubic graphs with the property that the minimum cardinality of connected cutsets is a logarithmic function with respect to the vertex size of the graphs, \emph{log length connected}.

\begin{rem}
It should be noted that there is a sharp contrast between the nets provided by the subcomplexes $\mathcal{C}_{sep}(S) \subset \mathcal{C}(S)$ and $\mathcal{P}_{sep}(S)\subset \mathcal{P}(S).$  Specifically,  regarding the curve complex, by topological considerations, it is immediate that the distance in the curve complex from any isotopy class of a simple closed curve to a non-trivially separating simple closed curve is bounded above by one, for all surfaces of finite type.  On the other hand, in the case of the pants complex, by Theorem \ref{thm:main2}, the maximal distance from an arbitrary pants decomposition to any pants decompositions containing a non-trivial separating curve is a non-trivial function depending on the topology of the surface.  In fact, for any infinite sequence of surfaces with a uniformly bounded number of boundary components, the function is unbounded.
\end{rem}

A key lemma used in the course of proving the lower bounds in Theorem \ref{thm:main2} and which may be of independent interest is the following:

\begin{lem}\label{lem:cut-set} (Key Lemma)
For $P \in \mathcal{P}(S)$ and $\Gamma(P)$ its pants decomposition graph, let $d$ be the cardinality of a minimal non-trivial connected cut-set $C \subset \Gamma(P)$.  Then
$$d_{\mathcal{P}(S)}(P,P') \geq \min \{ \mbox{girth}(\Gamma(P)),d\}-1$$
for $P'$ any pants decomposition containing a separating curve cutting off genus.
\end{lem}

The proof of Lemma \ref{lem:cut-set} brings together ideas related to the topology of the surfaces and graph theory in a simple yet elegant manner.

The results of this paper have some overlap with recent results Cavendish-Parlier \cite{cavendish} as well as \cite{rafitao}, the latter of which was posted to the arXiv subsequent to the posting of this article, regarding the asymptotics of the diameter of Moduli Space.  Although similar in nature, the results of this paper are in fact distinct from the aforementioned articles.  Specifically, due to the fact that the quasi-isometry constants of \cite{brock} between the pants complex and Teichm\"uller space equipped with the Weil-Petersson metric are dependent on the topology of the particular surface, the results of this paper are more properly related to complex of cubic graphs than to Moduli Space.  Accordingly, while the results of this paper can be used to consider the diameter of the complex of cubic graphs, they fail to provide direct information regarding the diameter of Moduli Space.  Conversely, while methods in \cite{cavendish} do contain lower bounds on the diameter of entire complex of cubic graphs, this paper focuses on the finer question of the density of a natural subset  inside the entire space.  On the other hand, it should be noted that methods in \cite{rafitao} do provide an independent and alternative (albeit non-constructive) proof of the lower bounds achieved in section \ref{sec:construction} of this paper by considering pants decompositions whose pants decomposition graphs are expanders.  Specifically, reliance on the existence of expander graphs provides for a potential alternative to the construction of log length connected graphs in Section \ref{sec:construction}.  Nonetheless, the explicit and constructive nature of the family of graphs in Section \ref{sec:construction} is a novelty of this paper as compared to \cite{rafitao}.   

The outline of the paper is as follows.  In Section \ref{sec:prelim} we introduce select concepts from graph theory and surface topology relevant to the development in this paper.  In Section \ref{sec:pantsgraph} we consider the pants decomposition graph of a pants decomposition of a surface.  The pants decomposition graph is a graph that is naturally associated to a pants decomposition of a surface which captures the topological type of the pants decomposition.  In Section \ref{sec:thm} we prove Theorem \ref{thm:main2} via a sequence of lemmas and corollaries.  The proof of Theorem \ref{thm:main2} in Section \ref{sec:thm} is complete modulo a construction of an infinite family of high girth, log length connected, at most cubic graphs, which is explicitly described in Section \ref{sec:construction}.  Finally in Section \ref{sec:ex}, the Appendix, some low complexity examples are considered.  
    
\subsection*{Acknowledgements}
$\\$
$\indent$
I want to express my gratitude to my advisors Jason Behrstock and Walter Neumann for their extremely helpful advice and insights throughout my research, and specifically with regard to this paper.  I want to further thank Jason for his thorough reading and comments on this work.  I would also like to acknowledge Maria Chudnovsky and Rumen Zarev for useful discussions regarding particular arguments in this paper.  

\section{Preliminaries} \label{sec:prelim}
\subsection{Graph Theory}
$\\$
$\indent$
Let $\Gamma=\Gamma(V,E)$ be an undirected graph with vertex set $V$ and edge set $E.$  The \emph{degree of a vertex} $v\in V,$ denoted $d(v),$ is the number of times that the vertex $v$ arises as an endpoint in $E.$  The \emph{degree of a graph} $\Gamma,$ denoted $d(\Gamma),$ is $\max\{ d(v) | v \in V\}.$  A graph $\Gamma$ is called \emph{k-regular} if each vertex $v\in V$ has degree exactly k.  In particular, 3-regular graphs are called \emph{cubic graphs}.  Furthermore, a graph $\Gamma$ is said to be \emph{at most cubic} if $d(\Gamma) \leq 3.$  

Given graphs, $\Gamma(V,E),\;  H(V',E'),$ H is called a \emph{subgraph} of $\Gamma,$ denoted $H \subset \Gamma,$ if $V' \subset V$ and $E' \subset E.$  In particular, for any subset $S\subset V(\Gamma),$ the \emph{complete subgraph of $S$ in $\Gamma$}, denoted $\Gamma[S]$, is the subgraph of $\Gamma$ with vertex set $S$ and edges between any pair of vertices $x,y \in S$ if and only if there is an edge $e\in E(\Gamma)$ connecting the vertices $x$ and $y.$  By definition $\Gamma[S] \subset \Gamma$.  As usual, we can make any graph $\Gamma$ into a metric space by endowing the graph with the usual \emph{graph metric}.  Specifically, we assign each edge to have length one, and then define the distance between any two vertices to be the length of the shortest path in the graph connecting the two vertices if the vertices are in the same connected component of $\Gamma,$ or infinity otherwise.  The \emph{diameter} of a  graph, denoted $diam(\Gamma),$ is the maximum of the distance function over all pairs of vertices in $\Gamma \times \Gamma.$  This diameter function can be restricted to subgraphs in the obvious manner.     

Given a graph $\Gamma,$ a \emph{walk} is a sequence of alternating vertices and edges, beginning and ending with a vertex, where each vertex is incident to both the edge that precedes it and the edge that follows it in the sequence.  The \emph{length} of a walk is the number of vertices in the walk.  A \emph{cycle} is a closed walk in which all edges and all vertices other than first and last are distinct.  A \emph{loop} is a cycle of length one.  A graph $\Gamma$ is \emph{acyclic} if it contains no cycles, i.e. its connected components are \emph{trees}.  The \emph{girth} of a graph $\Gamma$ is defined to be the length of a shortest cycle in $\Gamma,$ unless $\Gamma$ is acyclic, in which case the \emph{girth} is defined to be infinity.

A graph $\Gamma$ is \emph{connected} if there is a walk between any two vertices of the graph.  Otherwise, it is said to be \emph{disconnected}.  If a subset of vertices, $C\subset V,$ has the property that the \emph{deletion subgraph}, $\Gamma[V\setminus C]$, is disconnected, then $C$ is called a \emph{cut-set} of a graph.  If the deletion subgraph $\Gamma[V\setminus C]$, is disconnected and moreover it has at least two connected components each consisting of at least two vertices or a single vertex with a loop, $C$ is said to be a \emph{non-trivial cut-set}.  A (nontrivial) cut-set $C$ is called a \emph{minimal sized (non-trivial) cut-set} if $|C|$ is minimal over all (non-trivial) cut-sets of $\Gamma.$  On the other hand, a cut-set $C$ is said to be a \emph{minimal (non-trivial) connected cut-set} if $|C|$ is minimal over all (non-trivial) cut-sets $C$ of $\Gamma$ such that $\Gamma[C]$ is connected.  

In this paper we are interested in a family of graphs that are robust with regard to non-trivial disconnection by the removal of connected cut-sets.  More formally, we define an infinite family of graphs, $\Gamma_{i}(V_{i},E_{i})$, with increasing vertex size to be \emph{log length connected} if they have the property that the size of minimal non-trivial connected cut-sets of the graphs, asymptotically grows logarithmically in the vertex size of the graphs.  Specifically, if we set the function $f(i)$ to be equal to the cardinality of a minimal non-trivial connected cut-set of the graph $\Gamma_{i},$ then $f(i)= \Theta(\log(|V_{i}|)).$ 

\begin{ex}($(3,g)$-cages)
\label{ex:cage}
In the literature on graph theory, a family of graphs called \emph{$(3,g)$-cages} are a well studied, although not very well understood family of graphs.  By definition a $(k,g)$-cage is a graph of minimum vertex size among all $k$-regular graphs with girth $g.$  Note that $(k,g)$-cages need not be unique, and generally are not.  In \cite{erdos} it is shown that for $k\geq 2, g \geq 3,$ there exist $(k,g)$-cages.  Moreover if we let $\mu(g)$ represent the number of vertices in a $(3,g)$-cage, then it is well known that $2^{g/2} \leq \mu(g) \leq 2^{3g/4},$ see \cite{biggs}.  Furthermore, a theorem of Jiang and Mubayi, guarantees that the cardinality of a minimal non-trivial connected cut-set of a $(3,g)$-cage is at least $\lfloor  \frac{g}{2}  \rfloor.$  Combining the two previous sentences it follows that the family of $(3,g)$-cages are log length connected.  
\end{ex}
\subsection{Curve and Pants Complex}
$\\$
$\indent$ Given any surface of finite type, $S=S_{g,n},$ that is a genus $g$ surface with $n$ boundary components (or punctures), the \emph{complexity} of $S,$ denoted $\xi(S)\in \Z,$ is a topological invariant defined to be $3g-3+n.$  To be sure, while in terms of the $\MCG$ there is a distinction between boundary components of a surface and punctures on a surface, as elements of the $\MCG$ must fix the former, yet can permute the latter, for the purposes of this paper such a distinction is not relevant.  Accordingly, throughout this paper while we will always refer to surfaces with boundary components, the same results hold mutatis mutandis for surfaces with punctures.

A simple closed curve in $S$ is \emph{peripheral} if it bounds a disk containing at most one boundary component; a non-peripheral curve is \emph{essential}.  For $S$ any surface with positive complexity, the \emph{curve complex} of $S,$ denoted $\mathcal{C}(S),$ is the simplicial complex obtained by associating to each isotopy class of an essential simple closed curve a 0-cell, and more generally a k-cell to each unordered tuple $\{\gamma_{0}, ..., \gamma_{k} \}$ of $k+1$ isotopy classes of disjoint essential simple closed curves, or \emph{multicurves}.  This simplicial complex first defined by Harvey \cite{harvey} has many natural applications to the study of the $\MCG$ and is a well studied complex in geometric group theory.

Among simple closed curves on a surface of finite type we differentiate between two types of curves.  Specifically, a simple closed curve $\gamma \subset S$ is called a \emph{non-trivially separating curve}, or simply a \emph{separating curve}, if $S \setminus \gamma$ consists of two connected components $Y_{1}$ and $Y_{2}$ such that  $\xi(Y_{i}) \geq 1.$  Any other simple closed curve is \emph{non-separating}.  It should be stressed that, perhaps counterintuitively, a \emph{trivially separating curve}, that is a simple closed curve that cuts off two boundary components of the surface, under our definition, is considered a non-separating curve.  In light of the dichotomy between separating curves and non-separating curves, there is an important natural subcomplex of the curve complex called the \emph{complex of separating curves}, denoted $\mathcal{C}_{sep}(S),$ which is the restriction of the curve complex to the set of separating curves. 

For $S$ a surface of positive complexity, a \emph{pair of pants decomposition}, or simply a \emph{pants decomposition}, $P$ is a multicurve of maximal cardinality.  Equivalently, a pants decomposition $ 
P$ is a set of disjoint homotopically distinct curves such that the complement $S  \setminus P$ consists of a disjoint union of topological \emph{pairs of pants}, or spheres with three boundary components.  

Related to the curve complex, $\mathcal{C}(S),$ there is another natural complex associated to any surface of finite type with positive complexity: the \emph{pants complex}.  In particular, the 1-skeleton of the pants complex, the \emph{pants graph}, denoted $\mathcal{P}(S),$  is a graph with vertices corresponding to different pants decompositions of the surface, and edges between two vertices when the two corresponding pants decompositions differ by a so called \emph{elementary pants move}.  Specifically, two pants decompositions of a surface differ by an elementary pants move, if the two decompositions differ in exactly one curve and those differing curves intersect minimally inside the unique complexity one component of the surface, topologically either an $S_{0,4}$ or an $S_{1,1},$ in the complement of all the other agreeing curves in the pants decompositions.  By a theorem of Hatcher and Thurston, \cite{hatcher}, the pants graph is connected, and hence we have a notion of distance between different vertices, or pants decompositions $P_{1},P_{2} \in \mathcal{P}(S),$ obtained by endowing $\mathcal{P}(S)$ with the graph metric.  We denote this distance by  $d_{\mathcal{P}}(P_{1}, P_{2}).$  

Just as with the curve complex, there is an important subcomplex of the pants complex called the \emph{pants complex of separating curves}, denoted $\mathcal{P}_{sep}(S),$ which is the restriction of the pants graph to the set of those pants decompositions that contain a separating curve.  This paper analyzes the net of the pants complex of separating curves in the entire pants complex, for all surfaces of finite type.  

\section{Pants Decomposition Graph} \label{sec:pantsgraph}
By elementary topological considerations, it follows that for any pants decomposition $P \in P(S_{g,n}),$ the number of curves in the pants decomposition $P$, is equal to $\xi(S)=3g-3+n,$ while the number of pairs of pants into which the pants decomposition decomposes the surface is equal to $2(g-1)+n.$  Corresponding to any pants decomposition $P$ we define its \emph{pants decomposition graph}, $\Gamma(P),$ as follows: For $P \in \mathcal{P}(S),$ $\Gamma(P)$ is a graph with vertices corresponding the connected components of $S \setminus P,$ and edges between vertices corresponding to connected components that share a common boundary curve.  See Figure \ref{fig:pants1} for an example of a pants decomposition graph.  Pants decomposition graphs classify pants decompositions up to topological type.  Specifically, two pants decompositions have the same pants decomposition graph if and only if they divide the surface in the same topological manner, or equivalently the two pants decompositions differ by an element of the mapping class group.  

\begin{figure}[htpb]
\centering
\includegraphics[height=4 cm]{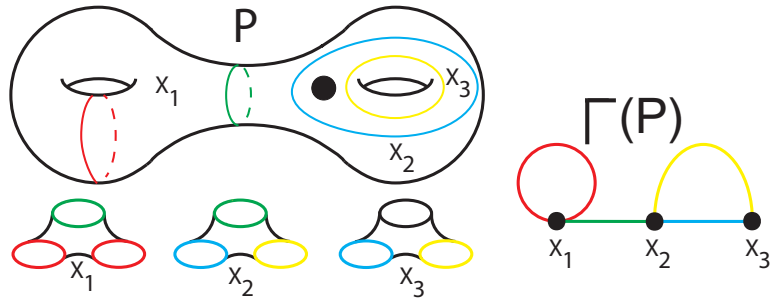}
\caption{$\Gamma(P)$ for $P \in P(S_{2,1})$. }\label{fig:pants1}
\end{figure}

\begin{rem}  The notion of pants decomposition graphs is considered in \cite{buser} as well as in \cite{parlier}.  Moreover, replacing the vertices in a \emph{pants decomposition graph} with edges and vice versa yields the \emph{adjacency graph} of Behrstock and Margalit \cite{behrstockmargalit} developed in the course of proving that the mapping class group is co-Hopfian with regard to finite index subgroups.
\end{rem}

The following elementary lemma, whose proof follows immediately, organizes elementary properties of $\Gamma(P)$ and gives a one to one correspondence between certain graphs and pants decomposition graphs:
\begin{lem}\label{lem:properties}
For $P \in \mathcal{P}(S_{g,n}),$ and $\Gamma(P)$ its pants decomposition graph:
\begin{enumerate}
\item $\Gamma(P)$ is a connected graph with $2(g-1)+n$ vertices and $3(g-1)+n$ edges
\item $\Gamma(P)$ is at most cubic
\end{enumerate}
Moreover, for all $q,p \in \N$, given any connected, at most cubic graph $\Gamma=\Gamma(V,E)$ with $|V|=2(p-1)+q$ and $|E|=3(p-1)+q,$ there exists a pants decomposition $P \in \mathcal{P}(S_{p,q})$ with pants decomposition graph $\Gamma(P) \cong \Gamma.$   
\end{lem}

Euler characteristic considerations imply the following corollary of Lemma \ref{lem:properties}:
\begin{cor} \label{cor:fundamental}
For $P \in \mathcal{P}(S_{g,n}),$ $\pi_{1}(\Gamma(P))$ is the free group of rank $g.$
\end{cor}
Another relevant elementary lemma is the following:
\begin{lem} \label{lem:projpants}
 Let $P \in \mathcal{P}(S_{g,n}),$ and let $\pi_{\mathcal{C}} \colon \mathcal{C}(S_{g,n}) \surj \mathcal{C}(S_{g,n-1}) \cup \emptyset$ be a projection map which fills in a boundary component.  Then the map $\pi$ extends to a surjection $$ \pi_{\mathcal{P}} \colon \mathcal{P}(S_{g,n}) \surj \mathcal{P}(S_{g,n-1}).$$  
 \end{lem}
 
 \begin{rem}
 Note that the map $\pi_{\mathcal{C}}$ has range $\mathcal{C}(S_{g,n-1}) \cup \emptyset$ as an essential curve that cuts off a pair of boundary components can become peripheral in the event that one of the cut off boundary components is filled in.
 \end{rem}
 
\begin{proof}
Under the map $\pi_{\mathcal{P}},$ all but one of the pairs of pants in a pants decomposition of $S_{g,n}$ are left unaffected.  The one affected pair of pants, which contains the boundary component being filled, becomes an annulus in $S_{g,n-1}.$  After identifying the two isotopic boundary curves of the annulus in $S_{g,n-1},$ we have a pants decomposition of $S_{g,n-1}.$  The fact that the projection $\pi_{\mathcal{P}}$ is surjective follows the observation that given any pants decomposition of $S_{g,n-1},$ one can easily construct a lift under $\pi_{\mathcal{P}}$ of the pants decomposition in $S_{g,n}.$
\end{proof}
In the next three subsections we explore certain aspects of pants decomposition graphs.

\subsection{Calculus of elementary pants moves and their action on pants decomposition graphs.}  \label{subsec:calc} Recall that there are two types of elementary pants moves depending on the type of complexity one piece in which the move takes place: 
\begin{description}
\item[E1] Inside a $S_{1,1}$ component of the surface in the complement of all of the pants curves except $\alpha,$ the curve $\alpha$ is replaced with $\beta$ where $\alpha$ and $\beta$ intersect once.
\item[E2]  Inside a $S_{0,4}$ component of the surface in the complement of all of the pants curves except $\alpha,$ the curve $\alpha$ is replaced with $\beta$ where $\alpha$ and $\beta$ intersect twice. 
\end{description}
Elementary move E1 has a trivial action on the pants decomposition graph $\Gamma(P),$ while the impact of the elementary move E2 can be described as follows: identify any two adjacent vertices, $v_{1},v_{2}$ in the pants decomposition graph connected by an edge $e,$ then the action of an elementary move E2 on the pants decomposition graph has the effect of interchanging any edge other than $e$ impacting $v_{1},$ or possibly the empty set, with any edge other than $e,$ impacting $v_{2},$ or possibly the empty set.  The one stipulation is that in the event that the empty set is being interchanged with an edge, the result of the action must yield a connected at most cubic graph.  An example of the action is presented in Figure \ref{fig:pantsmove}.  

\begin{figure}[bhtp]
\centering
\includegraphics[height=4 cm]{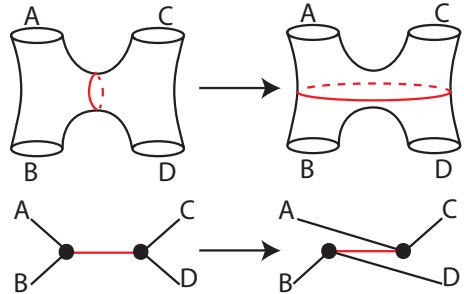}
\caption{An example of the action of an elementary pants move E2 on the pants decomposition graph.}\label{fig:pantsmove}
\end{figure}

\subsection{Adding boundary components}   \label{subsec:addingpunctures} 
Along the lines of the proof of Lemma \ref{lem:projpants}, note that any pants decomposition of $S_{g,n+1}$ can be obtained by beginning with a suitable pants decomposition of $S_{g,n}$, adding a boundary component appropriately, and then appropriately completing the resulting multicurve into a pants decomposition of $S_{g,n+1}$.  The effect that this process of adding a boundary component has on the pants decomposition graph has two forms, depending on whether topological pair of pants to which the boundary component is being added contains a boundary component of the ambient surface or not, as well as the manner in which the multicurve is completed into a pants decomposition of the resulting surface.  The two forms are depicted in Figure \ref{fig:addpuncture}.  

\begin{figure}[htpb]
\centering
\includegraphics[height=3.5 cm]{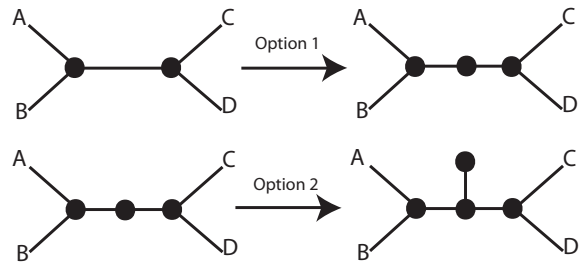}
\caption{Adding a boundary component to a pants decomposition graph has two possible forms.  In one case it adds a valence two vertex to the pants decomposition graph along an edge, while in the other case it adds a valence one vertex to the pants decomposition graph.}\label{fig:addpuncture}
\end{figure}

\subsection{Separating curves and pants decomposition graphs.} 
Given a pants decomposition $P \in \mathcal{P}(S),$ examining its pants decomposition graph $\Gamma(P)$ provides an easy way to determine if a pants decomposition $P$ contains a separating curve.  Specifically, a curve in a pants decomposition is a separating curve of the surface if and only if the effect of removing the corresponding edge in $\Gamma(P)$ non-trivially separates the graph into two connected components.  Recall that a non-trivial separation of a graph is a separation such that there are at least two connected components each consisting of at least two vertices or a single vertex and a loop.  

It is useful to differentiate two categories of separating curves,  
\begin{description}
\item[S1] separating curves that \emph{cut off genus},
\item[S2] and separating curves that \emph{cut off boundary components}.
\end{description}
By the former, we refer to separating curves on the surface whose removal separates that surface into two non-trivial subsurfaces each with genus at least one.  By the latter, we refer to  to separating curves on the surface whose removal separates that surface into two non-trivial subsurfaces at least one of which is a topological sphere with boundary components.  Equivalently, a separating curve $\gamma \in P \in \mathcal{P}(S)$ cuts off genus if the removal of the edge corresponding to $\gamma$ in $\Gamma(P)$ disconnects the graph into two cyclic components, otherwise if at least one of the connected components of $\Gamma(P) \setminus \gamma$ is acyclic, then the separating curve $\gamma$ cuts off boundary components.  Tracing through the definitions, it is immediate that separating curves that cut off genus can only exist on surfaces with genus at least two, while separating curves that cut off boundary components can only exist on surfaces with at least three boundary components.

\section{Proof of Theorem \ref{thm:main2}} \label{sec:thm}
In this section we prove the following technical theorem which in particular implies the statement of theorem \ref{thm:main2}: 
\begin{thm} \label{thm:main} (Main Theorem)
Let $S=S_{g,n}$ and set $D_{g,n} = \max_{_{P  \in \mathcal{P}(S) }} (d_{\mathcal{P}(S)} (P, \mathcal{P}_{sep}(S) )).$  Then, \begin{eqnarray*}
D_{g,n} &=& 0  \;\;\; \mbox{ for } g=0, n\geq 7  \\
& =& 1 \;\;\; \mbox{ for } g=0, n=6 \\
&=& 2 \;\;\; \mbox{ for } g=1, n\geq 3 \\ 
&\leq &    \lfloor 2\log_{2}(g-1) + 3 \rfloor \;\;\;   \mbox{ for } g \geq 2, n \leq 2 \\
&\leq & \min \left(  \lfloor 2\log_{2}(g-1) + 3 \rfloor,  \;\;  \lfloor \frac{16(g-1)}{ n} + 12 \rfloor   \right) \mbox{ for } g \geq 2, n \geq 3 \
\end{eqnarray*}
Furthermore, for any fixed number of boundary components (or punctures) $n$, $D_{g,n}$ grows asymptotically like the function $\log(g),$ that is $D_{g,n} = \Theta(\log(g)).$  On the other hand, for any fixed genus $g\geq 2, \;\forall n\geq 6g-5,$ $D_{g,n} = 2.$
\end{thm}

\begin{note}
For surfaces of \emph{low complexity}, i.e. $\xi(S) \leq 2,$ there are no non-trivially separating curves.  Hence, such surfaces are not included in Theorem \ref{thm:main}. 
\end{note}

The proof of the theorem is broken down into subcases which we prove as lemmas and corollaries.  The following is an outline of this proof.  First we prove the theorem for the special cases of genus zero and genus one surfaces.  Next, we consider the case of a fixed genus $g \geq 2$ surface, with a relatively large number of boundary components.  Finally, after developing some more general upper and lower bounds for the $g \geq 2$ cases, we prove the remainder of the theorem.  A portion of the proof depends on the existence of a family of special graphs constructed in Section \ref{sec:construction}.      

Recall that $D_{g,n} = \max_{_{P  \in \mathcal{P}(S_{g,n}) }} (d_{\mathcal{P}(S_{g,n})} (P, \mathcal{P}_{sep}(S_{g,n}) )).$  We begin by proving the genus zero case of Theorem \ref{thm:main}.  

\begin{lem} \label{lem:g0}
$D_{0,6}=1.$  More generally, for $n\geq 7,$ $D_{0,n}=0.$ 
\end{lem}

\begin{proof}
For the surface $S_{0,6},$ by Lemma \ref{lem:properties} and Corollary \ref{cor:fundamental}, a pants decomposition graph is a connected at most cubic tree with four vertices and three edges.  Up to isomorphism there are only two options, as presented in the left side of Figure \ref{fig:s06}.  By inspection, the claim of the lemma holds for $S_{0,6}$.  
\begin{figure}[htpb]
\centering
\includegraphics[height=2.25 cm]{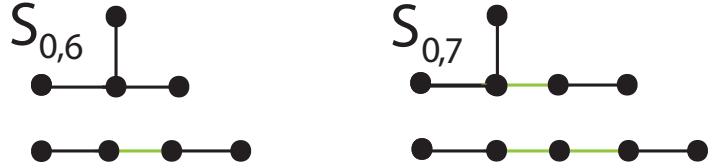}
\caption{Pants decomposition graphs of $S_{0,6}$ and $S_{0,7},$ respectively.  Green edges correspond to separating curves.}  \label{fig:s06} 
\end{figure}

Similarly, for the case of $S_{0,7}$ up to isomorphism, there are two pants decompositions graphs.  Both graphs contain separating curves, as shown in the right side of Figure \ref{fig:s06}.  More generally, as in subsection \ref{subsec:addingpunctures} for surfaces $S_{0,n}$ with $n> 7,$ any pants decomposition graph is achieved by appropriately adding boundary components to an appropriate pants decomposition graph of $S_{0,7}$.  Hence, the claim of the lemma holds from the immediate observation that the process of adding boundary components to a pants decomposition containing a separating curve yields a pants decomposition containing a separating curve.
\end{proof}

In the next lemma, we consider the genus one case of Theorem \ref{thm:main}.
\begin{lem} \label{lem:g1}
Assume $n \geq 3,$ then $D_{1,n}=2.$
\end{lem}

\begin{proof}
For the surface $S_{1,n},$ by  Lemma \ref{lem:properties} and Corollary \ref{cor:fundamental} any pants decomposition graph $\Gamma(P)$ is a connected at most cubic unicyclic graph with $n$ vertices and $n$ edges, where $n\geq 3.$  As such, there are three options for the isomorphism class of the pants decomposition graph $\Gamma$: 
\begin{enumerate}
\item $\Gamma$ contains a separating curve,
\item $\Gamma$ is an ($n-j$)-gon with $j \geq 1$ of the vertices of the ($n-j$)-gon having an additional edge connecting the vertex to a a valence one vertex,  or 
\item (3) $\Gamma$ is an n-gon.  
\end{enumerate}
See Figure \ref{fig:s1n} for the possibilities.  By inspection, the pants decomposition graphs of cases (1), (2), and (3) are distance zero, one, and two, respectively, from pants decompositions containing a separating curve.
\end{proof}

\begin{figure}[htpb]
\centering
\includegraphics[height=3 cm]{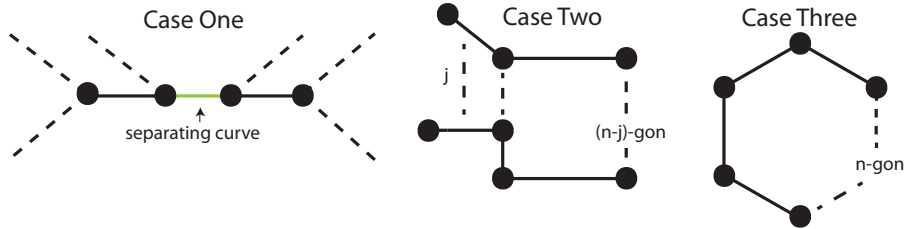}
\caption{Cases for pants decomposition graphs of $S_{1,n}.$ }\label{fig:s1n}
\end{figure}

In the next lemma, we describe a local situation in $\Gamma(P)$ which can be manipulated via elementary pants moves to generate a pants decomposition containing a separating curve.  

\begin{lem}\label{lem:5move} 
For $P \in \mathcal{P}(S)$ and $\Gamma(P)$ its pants decomposition graph.  If $\Gamma(P)$ has three consecutive vertices of degree at most two, then $d_{\mathcal{P}}(P,\mathcal{P}_{sep}) \leq 2.$
\end{lem}
\begin{proof}
It suffices to explicitly exhibit a process of two elementary pants moves for locally constructing a separating curve that cuts off boundary components assuming that $\Gamma(P)$ has three consecutive vertices of degree at most two.  See Figure \ref{fig:valence2} for these moves.
\end{proof}

\begin{figure}[htpb]
\centering
\includegraphics[height=2.4 cm]{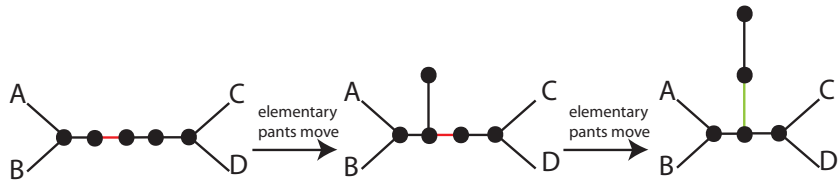}
\caption{Two elementary pants moves creating a separating curve that cuts off boundary components in $\Gamma$ beginning from a pants decomposition graph with three consecutive valence at most two vertices. }\label{fig:valence2}
\end{figure}

Using Lemma \ref{lem:5move} we have the following corollary, providing a sharp upper bound on $D_{g,n}$ for fixed $g\geq 2.$
\begin{cor} \label{cor:manypunctures1}
For all  $g\geq 2,$ $n\geq 6g-5$ $\implies D_{g,n}= 2.$
\end{cor}

\begin{proof}
We prove the lemma in two steps.
\begin{enumerate}
\item Step One:  $D_{g,n} \leq 2.$
$\\$ $\indent$
By Lemma \ref{lem:properties} for $P \in \mathcal{P}(S_{g,n})$, $\Gamma(P)$ is a connected at most cubic graph with $2(g-1)+n$ vertices and $3(g-1)+n$ edges.  Since $n\geq 6g-5,$ by pigeon hole considerations it follows that $\Gamma(P)$ has three consecutive vertices of degree at most two.  The proof of step one follows by Lemma \ref{lem:5move}.

\item Step Two: $D_{g,n} \geq 2.$
$\\$ $\indent$ By Lemma \ref{lem:properties} it suffices to explicitly exhibit connected at most cubic graphs with $2(g-1)+n$ vertices and $3(g-1)+n$ edges for all $g \geq 2, n \geq 6g-5$ such that the graphs neither contain non-trivial cut edges nor are one elementary move away from a graph with a non-trivial cut edge.  See Figure \ref{fig:distance2cut} for an explicit construction of such a family of graphs.
\end{enumerate}
\end{proof}

\begin{figure}[htpb]
\centering
\includegraphics[height=4 cm]{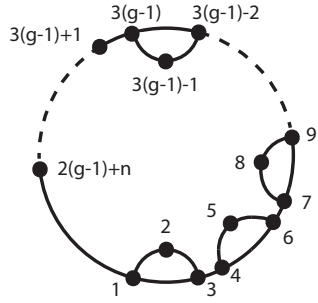}
\caption{A family of connected at most 3-regular graphs with $2(g-1)+n$ vertices and $3(g-1)+n$ edges, for all $g \geq 2, n \geq 6g-5.$ Such pants decompositions with corresponding graphs are distance (at least) two from a pair of pants containing a separating curve.}\label{fig:distance2cut}
\end{figure}

More generally, we have the following corollary providing an upper bound for $D_{g,n}$ based on local moves that create separating curves cutting off boundary components:

\begin{cor} \label{cor:manypunctures2}
Let $S=S_{g,n}$ with $g\geq 2,$  $n\geq 3$ $\implies  D_{g,n} \leq \lfloor \frac{16(g-1)}{n} + 12 \rfloor.$
\end{cor}
\begin{rem} It is likely that the constants in Corollary \ref{cor:manypunctures2} are not sharp.  Nonetheless they are needed for our current proof.
\end{rem}

\begin{proof}
Let $P$ be a pants decomposition of $S_{g,n}$, and let $\Gamma=\Gamma(P)$ be its pants decomposition graph.  By Lemma \ref{lem:properties}, $\Gamma$ is a connected at most cubic graph with $2(g-1)+n$ vertices and $3(g-1)+n$ edges.  Setting $V_{2}=\{v_{i} \in V(\Gamma) | d(v_{i}) \leq 2 \}$, and letting $V'_{2}$ be the set of vertices in $V_{2}$ with degree one vertices $v_{i}$ double counted.  Note that $|V_{2}| \leq \lceil \frac{n}{2} \rceil$, while $|V'_{2}|=n.$ 

Recall that if a vertex $v_{i}$ of degree at most two is adjacent to a vertex $x$ of degree three (two), then an elementary pants move can be applied to $\Gamma$ which has the effect of making the vertex $x$ have degree two (one) at the cost of increasing the degree of $v_{i}$ by one, as in subsection \ref{subsec:calc}.  In other words, elementary moves can be used to shuffle boundary components of the surface between adjacent pairs of pants in a pants decomposition.  Hence, to prove the corollary, by Lemma \ref{lem:5move} it suffices to show that for some three vertices $v_{j},v_{k},v_{l} \in V'_{2},$ the following inequality holds:
\begin{eqnarray}\label{eq:upperbound}
d_{\Gamma}(v_{j},v_{k}) + d_{\Gamma}(v_{j},v_{l}) \leq \lfloor \frac{16(g-1)}{n} + 10 \rfloor
\end{eqnarray}

Assume that the degree at most two vertices $v_{i}$ are scattered amongst the graph $\Gamma$ such that for any three vertices $v_{j},v_{k},v_{l} \in V'_{2},$ we have $d_{\Gamma}(v_{j},v_{k}) + d_{\Gamma}(v_{j},v_{l}) \geq m$.  Based on the size of the graph $\Gamma,$ we will obtain an upper bound on $m$ of  $\lfloor \frac{16(g-1)}{n} + 10 \rfloor$.  Thereby proving equation \ref{eq:upperbound} and completing the proof of the corollary.  

For any fixed vertex $v_{j} \in V'_{2}$ consider the two closest (not necessarily unique) vertices $v_{k},v_{l} \in V'_{2}$.  Let $d_{\Gamma}(v_{j},v_{k})=m_{k}$ and $d_{\Gamma}(v_{j},v_{l})=m_{l}$.  By assumption $m_{k}+m_{l} \geq m.$  Without loss of generality, assume that $m_{l}\geq m_{k}$ and hence $m_{l} \geq \lceil \frac{m}{2} \rceil.$  By construction, the first $\lfloor \left( \frac{\lceil \frac{m}{2} \rceil}{2} \right) \rfloor$ vertices traversed in a geodesic in $\Gamma$ from $v_{j}$ to $v_{l}$, including the initial vertex $v_{j}$, is disjoint from the first $\lfloor \left( \frac{\lceil \frac{m}{2} \rceil}{2} \right) \rfloor$ vertices of any similarly constructed geodesic with a different initial vertex $v_{i} \ne v_{j},v_{k}$ (see Figure \ref{fig:disjointpaths} for an illustration).  Furthermore, for the special cases of $v_{j}=v_{k}$ or $v_{k}=v_{l}$, namely where either $v_{j}$ or $v_{k}$ has degree one, the first $\lceil \frac{m}{2} \rceil$ vertices traversed in a geodesic in $\Gamma$ from $v_{j}$ to $v_{l}$, including the initial vertex $v_{j}$, is disjoint from all similarly constructed paths with different initial vertex, as well as from all previously constructed geodesic path subsegments. Putting things together and comparing with the total number of vertices in $\Gamma$, it follows that:  
$$ \lceil \frac{n}{2} \rceil \cdot \lfloor \left( \frac{\lceil \frac{m}{2} \rceil}{2} \right) \rfloor \leq 2(g-1)+n  \implies m \leq \frac{16(g-1)}{n}+10$$
\end{proof}

\begin{figure}[htpb]
\centering
\includegraphics[height=5 cm]{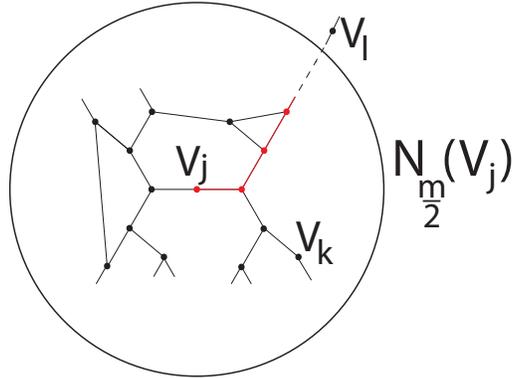}
\caption{The first $\lfloor \left( \frac{\lceil \frac{m}{2} \rceil}{2} \right) \rfloor$ vertices traversed in a geodesic in $\Gamma$ from $v_{j}$ to $v_{l}$, denoted in red, is disjoint from the first $\lfloor \left( \frac{\lceil \frac{m}{2} \rceil}{2} \right) \rfloor$ vertices of any similarly constructed geodesic with a different initial vertex $v_{i} \ne v_{j},v_{k}$.}\label{fig:disjointpaths}
\end{figure}

The following lemma shows that girth also provides an upper bound on the distance of a pants decomposition to a pants decomposition containing a separating curve.

\begin{lem} \label{lem:girth}
For $P \in \mathcal{P}(S)$ and $\Gamma(P)$ its pants decomposition graph  $$d_{\mathcal{P}}(P,\mathcal{P}_{sep}) \leq \mbox{girth}(\Gamma(P)) -1.$$
\end{lem}

\begin{proof}
Due to valence considerations a cycle of length one, or a loop, in the pants decomposition graph implies the corresponding pants decomposition contains a separating curve .  Hence, it suffices to show that given any cycle of length $n\geq 2,$ there exists an elementary pants move decreasing the length of a cycle by one.  Such an elementary pants move is represented in Figure \ref{fig:loop}. 
\end{proof}

\begin{figure}[htpb]
\centering
\includegraphics[height=3.5 cm]{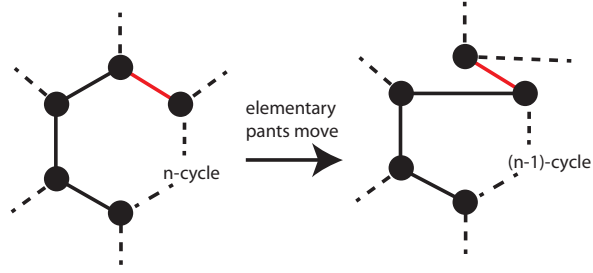}
\caption{Elementary pants move decreases the length of a cycle in $\Gamma.$}\label{fig:loop}
\end{figure}

As a corollary of Lemma \ref{lem:girth}, in conjunction with the discussion in Example \ref{ex:cage} which ensures that the girth of a cubic graph grows at most logarithmically in the vertex size of the graph, we have a logarithmic upper bound on $D_{g,0}$ for closed surfaces.  Specifically, we have the following corollary.  

\begin{cor} \label{cor:leqlogg} $\forall g \geq 2,$ 
$\; D_{g,n} \leq \lfloor2\log_{2}(g-1)+3\rfloor.$
\end{cor}

\begin{proof}
We begin with the case of closed surfaces.  By the discussion in Example \ref{ex:cage} regarding the number of vertices in a $(3,g)$-cage, it follows that for any cubic graph $\Gamma$ with $2(g-1)$ vertices, $$ \mbox{girth}(\Gamma)\leq \lfloor 2 \log_{2}(2(g-1)) \rfloor = \lfloor 2 \log_{2}(g-1)+2 \rfloor $$  
By Lemmas \ref{lem:properties} and \ref{lem:girth}, it follows that $D_{g,0} \leq \lfloor 2\log_{2}(g-1)+1 \rfloor .$  To complete the proof, it suffices to show that the process of adding $n$ boundary components as in subsection \ref{subsec:addingpunctures} to a closed surface cannot increase the distance to a separating curve by more than two elementary moves.  

The upper bound of $\lfloor 2\log_{2}(g-1)+1 \rfloor$ on the maximal distance to a pants decomposition containing a separating curve for closed surfaces is achieved by taking the smallest cycle $C$ in any graph $\Gamma(P)$ which has length at most $\lfloor 2\log_{2}(g-1)+2 \rfloor$ and then successively decreasing the length of cycle $C$ by elementary pants moves as in the proof of Lemma \ref{lem:girth}.  Consider what can happen to this cyclic subgraph as we add boundary components as in subsection \ref{subsec:addingpunctures}.  If the added boundary components do not affect the length of cycle $C,$ the upper bound is unaffected.  On the other hand, if the added boundary components increase the length of the cycle $C$ by adding one (two) degree two vertex (vertices) to the cycle $C$, then the distance to a separating curve increases by at most one (two).  However, once at least three degree two vertices have been added to the cycle $C,$ instead of reducing the cycle to a loop, we can instead use elementary moves to gather together three consecutive vertices of degree to obtain a pants decomposition with a separating curve cutting, as in Lemma \ref{lem:5move}.  Since in this situation the number of elementary moves needed to gather together at least three consecutive degree two vertices on $C$ is easily seen to be bounded above by one less than the length of the original cycle $C,$ the statement of the corollary follows in conjunction with the result of Lemma \ref{lem:5move}.  
\end{proof}

In the course of proving \ref{cor:leqlogg} we have in fact proven the following slight generalization of Lemma \ref{lem:girth} which proves useful in our consideration of low complexity examples in the appendix.  
\begin{cor} \label{cor:girth}
Let $P'\in \mathcal{P}(S_{g,n})$ be any pants decomposition obtained by adding boundary components to a pants decomposition $P \in \mathcal{P}(S_{g,m})$ for some $n >m$ as in subsection \ref{subsec:addingpunctures}.  Then the distance from $P'$ to a pants decomposition containing a separating curve is bounded above by the girth of $\Gamma(P)$ if $n=m+1$, or by one more than the girth of $\Gamma(P)$ for $n\geq m+2$.  
\end{cor}

Having developed upper bounds on the distance of pants decomposition to pants decomposition containing separating curves, presently we shift our focus to lower bounds.  Recall that a separating curve $\gamma \in \mathcal{C}_{sep}(S)$ is said to \emph{cut off genus} if $S \setminus \gamma$ consists of two connected complexity at least one subsurfaces neither of which is topologically a sphere with boundary components.  Also recall that for a graph $\Gamma(V,E)$, a subset $C \subset V$ is called a \emph{non-trivial connected cut-set} of $\Gamma$ if $\Gamma[C]$ is a connected graph and $\Gamma[V \setminus C]$ has at least two connected components each consisting of at least two vertices or a vertex and a loop.  The following lemma gives a lower bound on the distance of a pants decomposition to a pants decomposition which cuts off genus, in terms of the girth of the graph and the cardinality of a minimal non-trivial connected cut-set of the graph. 

\begin{lem}\label{lem:cut-set} (Key Lemma)
For $P \in \mathcal{P}(S)$ and $\Gamma(P)$ its pants decomposition graph, let $d$ be the cardinality of a minimal non-trivial connected cut-set $C \subset \Gamma(P)$.  Then
$$d_{\mathcal{P}(S)}(P,P') \geq \min \{ \mbox{girth}(\Gamma(P))-1,d-1 \}$$
for $P'$ any pants decomposition containing a separating curve cutting off genus.
\end{lem}

\begin{proof}
Let $\gamma$ be any curve in the pants decomposition $P$, and let $\alpha$ be any separating curve of the surface $S$ that cuts off genus.  It suffices to show that the number of elementary pants moves needed to take the curve $\gamma$ to $\alpha$ is at least  $\min \{ \mbox{girth}(\Gamma(P))-1, d-1 \}.$  In fact, considering the effect of an elementary pants move, it suffices to show that $\alpha$ non-trivially intersects at least  $\min \{ \mbox{girth}(\Gamma(P)), d \}$ different connected components of $S \setminus P$.

Corresponding to $\alpha$ consider the subgraph $[\alpha] \subset \Gamma(P)$ consisting of all vertices in $\Gamma(P)$ corresponding to connected components of $S \setminus P$ non-trivially intersected by $\alpha,$ as well as all edges in $\Gamma(P)$ corresponding to curves of the pants decomposition $P$ non-trivially intersected by $\alpha.$  By construction, the subgraph $[\alpha]$ is connected.  Note that the subgraph $[\alpha]$ need not be equal to the induced subgraph $\Gamma[\alpha],$ but may be a proper subgraph of it.  Nonetheless, $V(\Gamma[\alpha]) = V([\alpha]).$   (See Figure \ref{fig:cut-set} for an example of a subgraph  [a] $\subset \Gamma(P).$) 

As noted, it suffices to show $|V(\Gamma[\alpha])| \geq \min \{ \mbox{girth}(\Gamma(P)),\; d \}.$  Assume not, by the girth condition it follows that $\Gamma[\alpha]$ is acyclic.  However, this implies that $\alpha$ is entirely contained in a union connected components of $S\setminus P$ such that in the ambient surface $S,$ the connected components glue together to yield an essential subsurface $Y,$ which is topologically a sphere with boundary components.  Moreover, by the cardinality of the minimal non-trivial connected cut-set condition, it follows that the removal of the essential subsurface $Y,$ or any essential subsurface thereof, from the ambient surface $S$ does not, non-trivially separate $S$.  In particular, for all $U \subset Y,$ $S \setminus U$ consists of a disjoint union of at most one non-trivial essential subsurface as well as some number of pairs of pants.  It follows that $\alpha$ cannot be a separating curve cutting off genus.  
\end{proof} 

\begin{figure}[bhtp]
\centering
\includegraphics[height=2 cm]{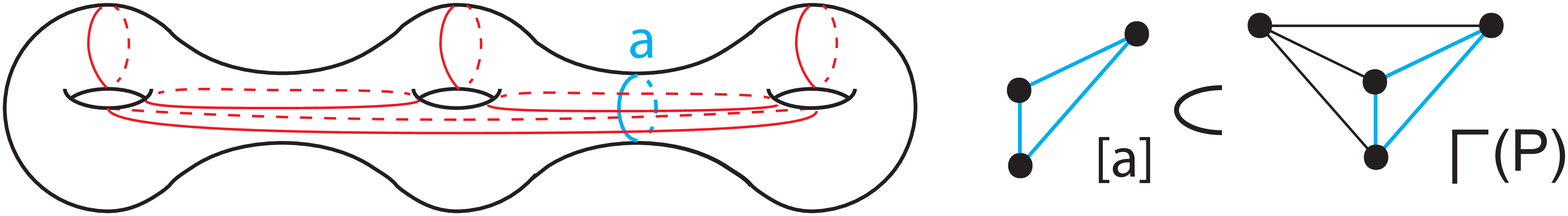}
\caption{An example of a subgraph $[a] \subset \Gamma(P)$ corresponding to a separating curve $a \subset S_{3,0},$ cutting off genus.  In this example, the girth of $\Gamma(P)$ is three and there are no non-trivial connected cut-sets of $\Gamma(P)$ .  Thus, by Lemma \ref{lem:cut-set}, the distance from $P$ to any pants decomposition with a separating curve cutting off genus is at least two.  In fact, it is not hard to see that the distance from $P$ to a pants decomposition $P'$ containing the curve $a$, which cuts off genus, is exactly two. }\label{fig:cut-set}
\end{figure}

An immediate consequence of Example \ref{ex:cage}, Corollary \ref{cor:leqlogg}, and Lemma \ref{lem:cut-set}, for an infinite family of pants decompositions $\{P_{m}\}_{m=1}^{\infty}$ of closed surfaces of genera $g_{m}$ whose pants decomposition graphs $\Gamma(P_{m})$ correspond to $(3,m)$-cages, it follows that $D_{g_{m},0}= \Theta(\log(g_{m})).$  It should be stressed however that because the number of vertices in $(3,m)$-cages grows exponentially, the family of $(3,m)$-cages cannot be used to prove the desired sharpness in Theorem \ref{thm:main}.  Furthermore, the family of $(3,m)$-cages are a highly non-constructive family of examples as to date outside of existence, little is known regarding $(3,m)$-cages for $m\geq 13$, \cite{exoo}.

In Section \ref{sec:construction} we produce a constructive family of 3-regular graphs, $\Gamma_{2m},$ in order to establish that the sharp asymptotic equality holds for all sequences of genera.  Specifically, for any even number $2m\geq 140$, such that $g$ is the largest integer satisfying $ \left(  \lceil \frac{2^{g}-4}{g-4}  \rceil \right) \cdot g \leq 2m,$ there exists a graph, $\Gamma_{2m},$ such that $|V(\Gamma_{2m})|=2m,$ girth$(\Gamma_{2m})=g,$ and any connected cut-set of the graph contains at least $ \lfloor \frac{g}{2} \rfloor$ vertices.  Furthermore, for any fixed number $n$ of boundary components, we can add $n$ boundary components to our graphs, $\Gamma_{2m},$ creating a family of pants decomposition graphs $\Gamma^{n}_{2m},$ whose corresponding pants decompositions similarly have girth, minimum non-trivial cut-set size, and distance between valence less than three vertices growing logarithmically in the vertex size of the graph.  By Lemma \ref{lem:cut-set}, the fact that girth and minimum non-trivial connected cut-set size grow logarithmically in the vertex size of the graph implies that the distance between pants decompositions with the given graphs as pants decomposition graphs to any pants decompositions containing a separating curve cutting off genus, grows logarithmically in the vertex size of the graph.  Moreover, the fact that the distance between valence less than three vertices grows logarithmically in the vertex size of the graphs, implies that the distance between pants decompositions with the given graphs as pants decomposition graphs and any pants decompositions containing a separating curve cutting off boundary components also grows logarithmically in the vertex size of the graphs.  Hence, as a corollary of the construction in Section \ref{sec:construction} we have:

\begin{cor} \label{cor:logplus}
Let  $n\in \N$ be fixed.  Then $D_{g,n}= \Theta(\log(g)).$  
\end{cor}

The proof of the main theorem, Theorem \ref{thm:main}, follows immediately from the combination of Lemmas \ref{lem:g0} and \ref{lem:g1} as well as Corollaries \ref{cor:manypunctures1}, \ref{cor:manypunctures2},  \ref{cor:leqlogg}, and \ref{cor:logplus}
\qed.

\section{Construction of Large Girth, Log Length Connected Graphs} \label{sec:construction}
We first describe a construction for a family, $\Gamma_{g},$ of 3-regular girth $g\geq5$ graphs with $$ \left( \lceil \frac{2^{g}-4}{g-4}\rceil  \right)  \cdot g + [ \left( \lceil \frac{2^{g}-4}{g-4}\rceil  \right)  \cdot  g] \text{(mod 2)}$$ vertices (where the final term is simply to ensure the total number of vertices is even), which have the property that any connected cut-set of $\Gamma_{g}$ contains at least $\lfloor\frac{g}{2} \rfloor$ vertices.  Afterward, we generalize our construction, interpolating between the family of graphs $\Gamma_{g}.$  Specifically, for all $m \in \N$, $m\geq 70$ such that $g\geq 5$ is the largest integer satisfying $2m \geq  \left(  \lceil \frac{2^{g}-4}{g-4}  \rceil \right) \cdot  g,$ there exists a 3-regular girth $g$ graph $\Gamma_{2m}$ with $2m$ vertices and the property that any connected cut-set of the graph contains at least $ \lfloor  \frac{g}{2} \rfloor$ vertices.  Finally, we demonstrate that for any fixed number of boundary components $n,$ we can add $n$ boundary components to our graphs $\Gamma_{2m}$ yielding a family of graphs $\Gamma^{n}_{2m}$ with girth, non-trivial minimum cut-set size, and the distance between valence less than three vertices growing logarithmically in the vertex size of the graph. 

\subsection{Construction of  $\Gamma_{g}$}
Begin with $ \left( \lceil \frac{2^{g}-4}{g-4}\rceil  \right) $ disjoint cycles each of length $g$ (possibly one of length g+1 if necessary to make the total number of vertices even).  Then, we chain together these disjoint cycles into an at most 3-regular connected tower $T_{g}$,  connecting each cycle to its neighboring cycle(s) by adding two edges between pairs of vertices, one from each cycle, such that each of the two vertices from the same cycle, to which edges are being attached, are of distance at least $\lfloor\frac{g}{2} \rfloor.$  See Figure \ref{fig:graphg2g} for an example in the case $g=8.$  

\begin{figure}[htpb]
\centering
\includegraphics[height=5 cm]{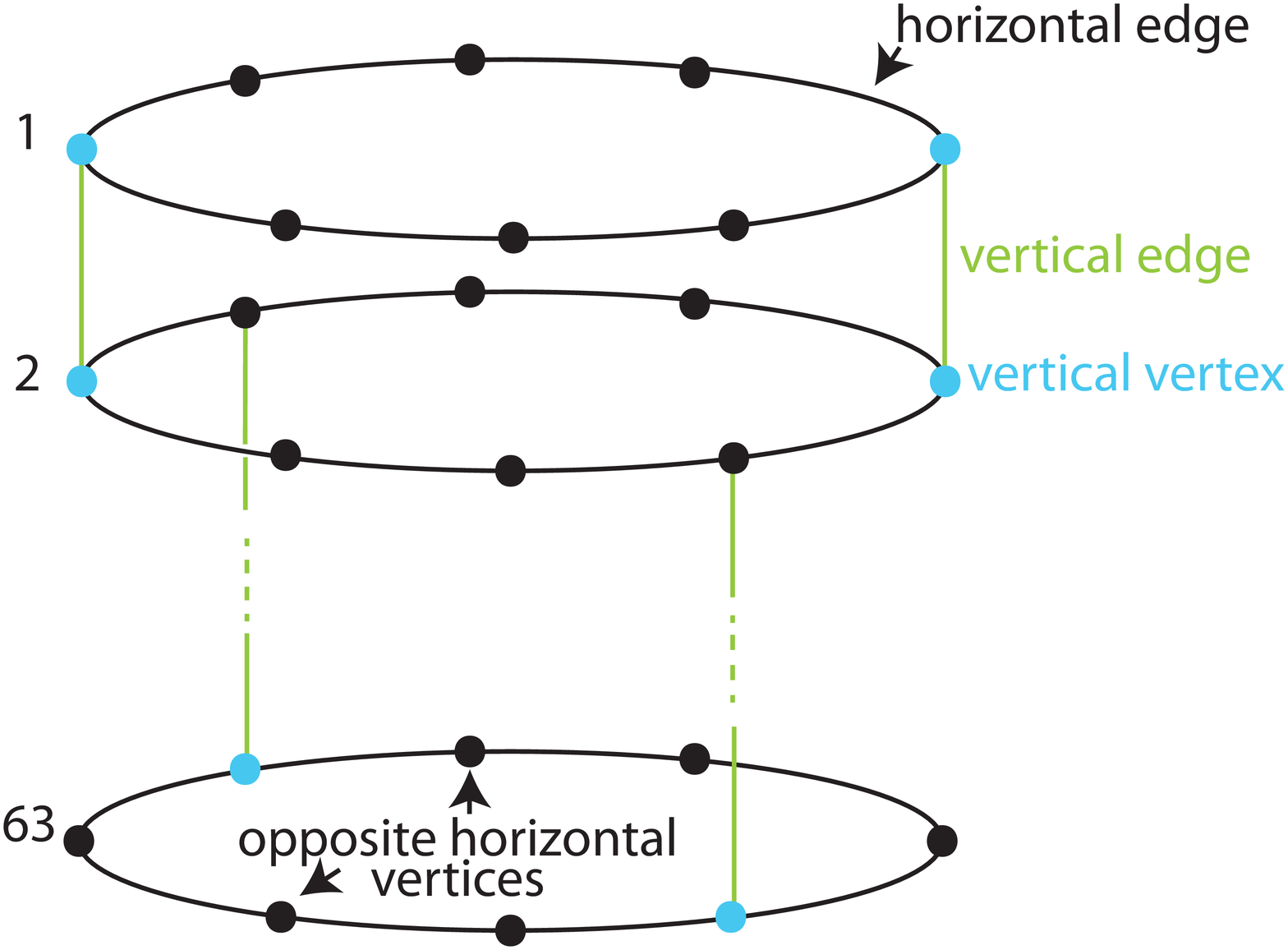}
\caption{$T_{8},$ an at most 3-regular girth eight tower graph with $|V|=\left( \lceil \frac{2^{8}-4}{8-4}\rceil  \right) \cdot 8=63 \cdot 8.$}\label{fig:graphg2g}
\end{figure}

As motivated by Figure \ref{fig:graphg2g}, we call edges of the original disjoint cycles \emph{horizontal edges} and edges that were added to complete it into a tower \emph{vertical edges}.  Vertices adjacent to vertical edges are called \emph{vertical vertices} and likewise for horizontal edges.  Two vertical edges between the same cycles are called \emph{opposite vertical edges} and their corresponding vertices to which opposite vertical edges are incident are called \emph{opposite vertical vertices}.  Finally, each pair of opposite vertical vertices on the same cycle gives rise to a \emph{partition} of the horizontal vertices of the given cycle corresponding to the connected components of the cycle in the complement of the vertical vertices.

In the following remark we record some observations regarding our towers, $T_{g}.$
\begin{rem} By construction, the tower graphs, $T_{g},$ constructed above have the following properties:
\renewcommand{\labelenumi}{(\alpha{enumi})}
\begin{description}
\item[T1]  $T_{g}$ has $ \left( \lceil \frac{2^{g}-4}{g-4}\rceil  \right)  \cdot  g + [ \left( \lceil \frac{2^{g}-4}{g-4}\rceil  \right)  \cdot g] $(mod 2) vertices.
\item[T2]  $T_{g}$ is an at most 3-regular and at least 2-regular graph with girth $g.$  
\item[T3]  If we denote the subset of vertices of $T_{g}$ of valence two by $V^{T_{g}}_{2},$ then $|V^{T_{g}}_{2}| \geq 2^{g}.$  
\item[T4]  Any connected cut-set of $T_{g}$ has at least $\lfloor\frac{g}{2} \rfloor$ vertices. 
\end{description}
\end{rem}

\subsection{Algorithm completing $T_{g}$ to a 3-regular graph $\Gamma_{g}$}
Presently we describe a constructive algorithm to add edges to the tower $T_{g}$ completing it to a 3-regular graph $\Gamma=\Gamma_{g}$ which also has girth $g,$ and retains the property that  any connected cut-set of $\Gamma_{g}$ has at least $\lfloor\frac{g}{2} \rfloor$ vertices.  The following process is motivated by a theorem from \cite{biggs}.  Throughout, by abuse of notation, we will always refer to the graph that has been constructed up to the current point as $\Gamma.$  In terms of ensuring the girth condition, the main observation, to be used implicitly throughout, is that removing edges from a graph never decreases girth, while adding an edge connecting vertices which were previously at least distance $g-1$ apart, in a girth at least $g$ graph, yields a girth at least $g$ graph.

\begin{description}
\item[Step One] An Easy Opportunity to Add an Edge
$\\$ $\indent$
If $\Gamma$ is 3-regular, we're done.  If not, fix a vertex $v \in V^{T_{g}}_{2}$ of valence two.  If there exists some other vertex $x \in V^{T_{g}}_{2}$ with $d_{\Gamma}(v,x)\geq g-1,$ add an edge between x and v.  

\item[Step Two]  Exhaust Easy Opportunities
$\\$ $\indent$
Iterate step one until all possibilities to add edges to $\Gamma$ are exhausted.

\item[Step Three] One Step Backward, Two Steps Forward
$\\$ $\indent$
If $\Gamma$ is 3-regular, we're done.  If not, since the total number of vertices is even, there must exist at least two vertices, $x$ and $y,$ of valence two.  Consider the sets $U=N^{\Gamma}_{g-2}(x) \cup N^{\Gamma}_{g-2}(y) $ and $I=N^{\Gamma}_{g-2}(x) \cap N^{\Gamma}_{g-2}(y).$  Due to the valence considerations, since $x,y$ are valence two vertices in an at most cubic graph it follows that $|N^{\Gamma}_{g-2}(x) |\leq 1 + 2 + ... +2^{g-2} = 2^{g-1}-1,$ and similarly for $N^{\Gamma}_{g-2}(y).$  Note that $|U| = |N^{\Gamma}_{g-2}(x) |+ |N^{\Gamma}_{g-2}(y) | - |I| \leq 2^{g}-2 - |I|.$  Then 
consider the set $W= V^{T_{g}}_{2} \setminus U.$  Since $|V^{T_{g}}_{2}| \geq 2^{g},$ it follows that $|W| \geq 2 + |I|.$  In particular, the set W is non-empty.  Furthermore, considering that step two was completed to exhaustion, it follows that $\forall w\in W,$ $w$ is of valence three in $\Gamma.$  Moreover, by definition, the vertex $w$ is of valence two in $T_{g}.$  Denote the vertex that is connected to $w$ in $\Gamma$ but not in $T_{g}$ by $w'.$  Perforce, $w'$ is distance at least $g-2$ from both $x$ and $y.$  In fact, we can assume that $w' $ is not exactly distance $g-2$ from both $x$ and $y$ because $|W| > |I|.$  For concreteness, assume $d_{\Gamma}(x,w')\geq g-1.$  

Remove from $\Gamma$ the edge $e$ connecting $w$ to $w',$ and in its place include two edges: $e_{1}$ between $x$ and $w',$ and $e_{2}$ between $w$ and $y.$  Adding the two edges $e_{1}$ and $e_{2}$ does not decrease girth to less than $g$ as they each connect vertices that were distance at least $g-1$ apart:  Because $\Gamma$ was girth at least $g,$ after removing $e,$ the vertices $w$ and $w'$ are distance at least $g-1.$  Hence, even after adding edge $e_{1}$ we can still be sure that the vertices $y$ and $w$ remain distance at least $g-1$ apart, thereby allowing us to add edge $e_{2}$ without decreasing girth to less than $g.$

\item[Step Four] Repeat
$\\$ $\indent$
If $\Gamma$ is not yet 3-regular, return to step three.
\end{description}
Note that the above completion algorithm necessarily terminates by induction as Step three can always be performed if the graph is not yet 3-regular, and the net effect of Step three is a to increase the number of edges in the at most 3-regular graph by one edge at a time.  Moreover, note that the algorithm never removes edges from the tower $T_{g}$, and hence the resulting graph $\Gamma_{g}$ includes the tower $T_{g}$ as a subgraph.  Since by construction the graph $\Gamma_{g}$ has girth $g,$ all that remains is to verify that the completed graph $\Gamma_{g}$ has the property that any connected cut-set has at least $\lfloor \frac{g}{2} \rfloor$ vertices.    

\begin{figure}[htpb]
\centering
\includegraphics[height=6 cm]{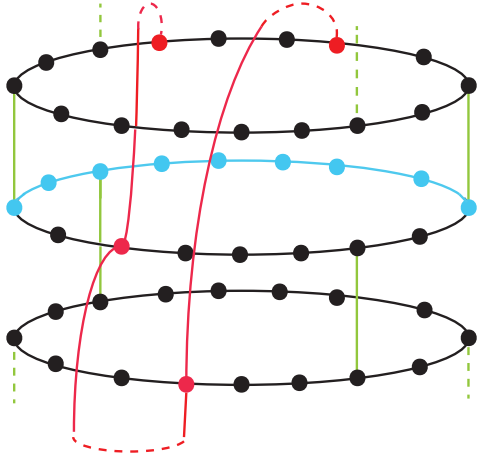}
\caption{Examples of connected cut-sets in $\Gamma_{g}.$  The red cut-set does not contain any vertical edges, but contains at least half the vertices of a cycle $\Gamma_{g}$ (not all the edges are drawn).  The blue cut-set contains vertices on two opposite vertical edges. } \label{fig:concutset}
\end{figure}

\begin{lem}Any connected cut-set of $\Gamma_{g}$ has at least $\lfloor \frac{g}{2} \rfloor$ vertices.  
\end{lem}
\begin{proof}
Let $C$ be a connected cut-set of $\Gamma_{g}$.  Without loss of generality we can assume that $C$ has the property that it does not contain any connected proper subcut-set.  In particular, due to valence consideration, $C$ cuts the tower into exactly two pieces.  

A couple of preliminary observations are in order.  Firstly, recall that by construction, $\Gamma_{g}$ contains $T_{g}$ as a subgraph.  Hence, $C$ must include vertices (not necessarily connected in the tower) that cut the subgraph $T_{g}.$  Secondly, in a girth $g$ graph any non-backtracking walk of length at most $\lfloor \frac{g}{2} \rfloor$ is a geodesic.  
%In particular, it follows that any connected cut-set in $\Gamma_{g}$ which contains vertices which are distance at most $\lfloor \frac{g}{2} \rfloor$ in the tower contains at least $\lfloor \frac{g}{2} \rfloor$ vertices.  

%Furthermore, any connected cut-set in $\Gamma_{g}$ containing vertices which are distance less than $\lfloor \frac{g}{2} \rfloor$ in the tower, either contains at least $\lfloor \frac{g}{2} \rfloor$ vertices of $\Gamma_{g}$ or an entire path of vertices in the tower connecting them.  

Consider the options for a cut-set of $T_{g}$.  Since $T_{g}$ is \emph{2-connected}, any cut-set of $T_{g}$ must contain at least two vertices.  There are two types of two-vertex cut-sets: either the cut-set contains two opposite vertical vertices and hence cuts the tower horizontally (see the blue cut-set in Figure \ref{fig:concutset}), or the cut-set cuts off a portion of a single horizontal cycle of $T_{g}$ from the rest of the tower (see the red cut-set in Figure \ref{fig:concutset}).  In the first case, the statement of the lemma holds because the two vertical vertices are distance $\lfloor \frac{g}{2} \rfloor$ apart in the tower and therefore in $\Gamma_{g}$ as well.  Similarly, in the second case, the desired property holds because the two vertices are distance less than $ \lfloor \frac{g}{2} \rfloor$ apart in the tower and hence any path connecting them in $\Gamma_{g}$ (disjoint from the segment being cut off) must have length at least $\lfloor \frac{g}{2} \rfloor.$

\begin{figure}[htpb]
\centering
\includegraphics[height=5 cm]{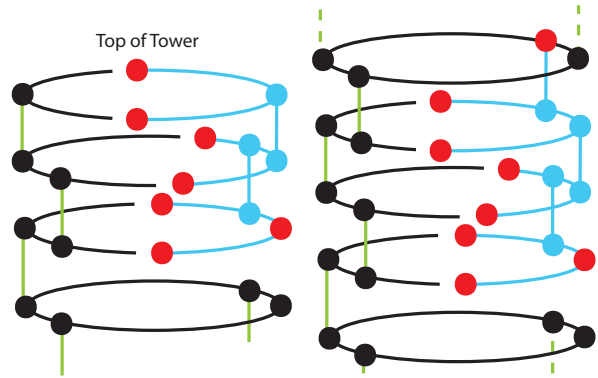}
\caption{Cut-sets of $\Gamma_{g}$  which never contains two opposite vertical vertices, and fails to cut off a portion of a single horizontal cycle of $T_{g}$ from the rest of the tower.   Any such connected cut-set either has at least $\lfloor \frac{g}{2} \rfloor$  vertices or is not actually a connected cut-set, but rather a connected subtree of the tower whose complement in the tower is connected.  } \label{fig:concutset4}
\end{figure}

To see that two aforementioned options for cut-sets are all we need to consider, assume to the contrary.  That is, assume there is a connected cut-set $C$ of $\Gamma_{g}$, with at most $\lfloor \frac{g}{2} \rfloor$ vertices, that doesn't contains two opposite vertical vertices, and fails to cut off a portion of a single horizontal cycle of $T_{g}$ from the rest of the tower.  Considering the connectivity of the tower $T_{g},$ which lies as a subgraph in $\Gamma_{g},$ it follows that  on a continuous sequence of horizontal cycles in the tower $T_{g}$ of differing heights, e.g. from height level $i$ to $j,$ the cut-set $C$ must contain pairs of horizontal vertices, in different partitions.  Additionally, unless the horizontal cycles $i$ and $j$ are either the top or bottom horizontal cycles of the tower, in each of these horizontal cycles the cut-set $C$ must contain a vertical vertex.  (See examples in Figure \ref{fig:concutset4}).  

In this case, since pairs of horizontal vertices on the same cycles are distance at most $ \lfloor \frac{g}{2} \rfloor$ apart in the tower, our vertex bound on $C$ ensures that $C$ contains an entire geodesic between the pair of horizontal vertices in the tower.  Accordingly, as the pairs of horizontal vertices in each horizontal cycle were in different partitions of the horizontal cycle it follows that $C$ contains a continuous sequence of vertical vertices for all height levels of the tower from $i$ to $j.$  Moreover, we can also assume that this continuous sequence of vertical vertices are all connected to each other, for otherwise $C$ would contain a pair of opposite vertical vertices which we assumed not to be the case.  It turns out that our desired connected cut-set $C$ is not in fact a cut-set of the tower, but rather a connected subtree of the tower whose complement in the tower $T_{g}$, and hence in the graph $\Gamma_{g},$ is connected. 
\end{proof}

\subsection{Construction of  $\Gamma_{2m}$}
Having completed the construction of the graphs $\Gamma_{g},$ presently we show that for any even number of vertices $2m$ such that $2m \geq V(\Gamma_{g})$, for some $g$, we can construct a 3-regular girth $g$ graph on $2m$ vertices, which we denote $\Gamma_{2m},$ with the property that any connected cut-set of $\Gamma_{2m}$ contains at least  $\lfloor \frac{g}{2} \rfloor$ vertices.  In fact, we can construct the graphs $\Gamma_{2m}$ using the exact same process as in the construction of $\Gamma_{g}$ with the exception that we now start with additional cycles in our initial tower which is subsequently completed to a cubic graph.  Specifically, to construct $\Gamma_{2m}$, we begin with $\lfloor \frac{2m}{g} \rfloor$ cycles of length $g$ and $(g+1)$ as necessary.  The conclusions in this case follow exactly as above.  

\subsection{Adding a fixed number $n$ of boundary components to $\Gamma_{2m}$}
For any fixed number $n$ of boundary components, by basic counting considerations we can add $n$ boundary components to our graphs, $\Gamma_{2m},$ yielding a family of graphs $\Gamma^{n}_{2m},$ such that the distance between the $n$ added boundary components grows logarithmically in the vertex size of the graphs $\Gamma^{n}_{2m}.$  Specifically, for a fixed number $n$ of added boundary components, past some minimal threshold for $2m$ we can easily ensure that no two added boundary components in $\Gamma_{2m}$ are within distance $\lfloor \frac{g}{2} \rfloor$ from each other.  This is because for $x,$ an added boundary component in $\Gamma_{2m},$ since $x$ is a degree two vertex in an at most cubic graph, $$|N_{\lfloor \frac{g}{2} \rfloor} (x)| \leq 2^{\lfloor \frac{g}{2} \rfloor +1},$$ while $|V(\Gamma_{2m})| \geq 2^{g}.$
\section{Appendix: Low Complexity Examples} \label{sec:ex}

\begin{note} \label{note:s}  \cite{sloane} [A002851] The number of connected simple (no loops or multiple edges) cubic graphs with 2n vertices for $n=1, 2, 3, 4, 5, 6, 7, 8, 9, ..$ is $0, 1, 2, 5, 19, 85, 509, 4060, 41301, ...$  Moreover, classifying the above graphs in terms of girth, we have the following table based on \cite{robinson}:
\begin{center}
\begin{tabular}{|c|c|c|c|c|c|c|}
\hline
Vertices & girth $\geq 3$ & girth $\geq 4$ & girth $\geq 5$ & girth $\geq 6$  \\ 
\hline   
4 & 1 & 0 & 0 &0 \\
\hline   
6 & 2 & 1 & 0 &0\\
\hline   
8 & 5 & 2  &0 &0\\
\hline   
10 & 19 & 6 & 1 &0 \\
\hline   
12 & 85 & 22 & 2 &0\\
\hline   
14 & 509 & 110 & 9 & 1 \\
\hline   
%16 & 4060 & 792 & 49 & 1 &0 &0\\
%\hline   
%18 & 41301 & 7805 & 455 & 5 &0 &0\\
%\hline
%20	& 510489	& 97546	&5783	&32&	0&	0 \\
%\hline
%22 &	 7319447 &	1435720& 	90938 & 	385	& 0	& 0\\
%\hline
%24	& 117940535	& 23780814 &	1620479	& 7574	&1&	0\\
%\hline
%26	& 2094480864	& 432757568 &	31478584	 & 181227	&3	&0\\
%\hline
%28	&40497138011	&?&	656783890	&4624501&	21	&0\\
%\hline
%30	&845480228069&	?	&?&	122090544	&546	&1\\
%\hline
\end{tabular} 
\end{center}

$\\$

\end{note}

The following is a table of some values of $D_{g,n}.$
\begin{center}
\begin{tabular}{|c||c|c|c|c|c|c|c|c|c|c|c|c|c|c|c|}
%\hline
%\hline
%$n\geq 43$ & $0$ & $2$ & $2$ &  $2$ & $2$ & $2$ & $2$ &$2$ &$2$\\
%\hline
%$\vdots$ & $\vdots$ & $\vdots$ & $\vdots$ &$\vdots$ & $\vdots$ & $\vdots$ &$\vdots$ & $\vdots$ & $\vdots$ \\
\hline
$8$ & $0$ & $2$ & $2$ &  $3$ & $4$ & $4$ & $4$ & $5$ & $5$ \\% &\multicolumn{5}{|c|}{ }  \\
\hline
$7$ & $0$ & $2$ & $2$  & $3$ & $4$ & $4$ & $4$ & $5$ & $6$ \\% &\multicolumn{5}{|c|}{ } \\
\hline
$6$ & $1$ & $2$ & $3$ & $4$  & $4$ & $4$ & $5$ & $5$ & $6$\\ % & \multicolumn{5}{|c|}{ } \\
\hline
$5$ &         & $2$  & $2$ & $4$ & $4$ & $4$ & $5$ & $5$ & $6$  \\ %&\multicolumn{5}{|c|}{ } \\
\hline
$4$ &         & $2$  & $2$ & $4$ & $4$ & $4$ & $5$ & $5$ & $6$   \\%&\multicolumn{5}{|c|}{ } \\
\hline
$3$ &         & $2$  & $2$ & $3$ & $4$ & $4$ & $4$ & $5$ & $5$ \\%  &\multicolumn{5}{|c|}{ } \\
\hline
$2$ &         &          & $1$ & $3$ & $3$ & $3$ & $4$ & $5$ & $5$   \\%&\multicolumn{5}{|c|}{ } \\
\hline
$1$ &         &          & $1$ & $2$ & $3$ & $3$ & $4$ & $4$ & $5$ \\ %\multicolumn{5}{|c|}{ } \\
\hline
$0$ &         &          & $1$ & $2$ & $3$ & $3$ & $4$ & $4$ & $5$ \\ %& $5$ & $5$ & $5$ & $5$ & $6$  \\
\hline
\hline
$n \uparrow \; g \rightarrow  $  & $0$  & $1$ & $2$  & $3$ & $4$ & $5$  & $6$ & $7$ & $8$ \\  %& $9$  & $10$ & $11$ & $12$ & $13$ \\
\hline   
\end{tabular} 
\end{center}

$\\$

The first two columns, namely genus zero and one, of the table are immediate by Lemmas \ref{lem:g0} and \ref{lem:g1}.  The first row, namely closed surfaces, follows from Lemmas \ref{lem:girth} and \ref{lem:cut-set} in conjunction with the data from Note \ref{note:s} regarding the number of vertices in $(3,m)$-cages.  The genus two column of the table follows from directly considering the two isomorphism classes of two vertex three regular graphs and the consequences of adding boundary components as in subsection \ref{subsec:addingpunctures}.  

The genus three column of the table arises from considering the graph $[2]^{4},$ in LCF notation, and then the consequences of adding boundary components.  By Note \ref{note:s}, this is the only simple connected cubic graph on four vertices.  Any non-simple graph has girth at most two and hence is distance at most one from a pants decomposition containing a separating curve.  Moreover, as in Corollary \ref{cor:girth} adding a single boundary component to any non-simple graph yields a graph that is distance at most two from a pants decomposition containing a separating curve whereas adding an arbitrary number of boundary components to a non-simple graph yields a graph that is distance at most three from a pants decomposition containing a separating curve.

The genus four column of the table arises from considering the graph $[3]^{6}=K_{3,3}$, or the $(3,4)$-cage.  By Note \ref{note:s}, there is only one other simple connected cubic graph on six vertices, namely the graph $[2,3,-2]^{2}$, which has girth three.  By Corollary \ref{cor:girth}, adding any number of boundary components to the graph $[2,3,-2]^{2}$ produces pants decomposition with distance at most four from a pants decomposition containing a separating curve.  In fact, since the graph $[2,3,-2]^{2}$  has two disjoint cycles of length three, it follows that no matter how boundary components are added to the graph, for all the cases in the table, one cannot produce examples of pants decomposition graphs that are further from a pants decomposition containing a separating curve than can be produced by adding boundary components to the graph $[3]^{6}$.

The genus five column of the table arises from considering the graphs $[-3,3]^{4}$ and $[4^{8}].$  By Note \ref{note:s}, in total there are five isomorphism classes of simple cubic graphs on eight vertices to consider.  Of the five graphs, the three graphs with girth three can be ignored as two of the graphs contain a pair of disjoint cycles of length three, while the third graph contains disjoint cycles of lengths three and four.  It follows that if less than three boundary components are added to these graphs, the distance to a pants decomposition containing a separating curve is bounded above by three.  Moreover, by Corollary \ref{cor:girth}, adding any number of boundary components to such graphs produces pants decompositions which are distance at most four from a pants decomposition containing a separating curve.  Hence, we can ignore the three girth three graphs. 

The genus six column of the table arises from considering the Petersen graph, or the $(3,5)$-cage.  By Note \ref{note:s}, excluding the Petersen graph, in total there are eighteen other isomorphism classes of simple cubic graphs on ten vertices to consider.  However, thirteen of the graphs have girth three and hence by Corollary \ref{cor:girth}, adding any number of boundary components to a girth three graph produces a graph that is at most distance four from a pants decomposition containing a separating curve and hence can be ignored.  Next, direct consideration of the five remaining girth four graphs on ten vertices reveals that they each contain pairs of disjoint cycles of length four.  It follows for all the cases in the table, adding boundary components to these graphs cannot produce examples of graphs that are further from a pants decomposition containing a separating curve than graphs that can be produced by adding boundary components to the Petersen graph.  

The genus seven column of the table arises from considering the  graph $$[-4,5,-4,4,-5,4,-5,-4,4,-4,4,5]$$   By Note \ref{note:s}, in total there are 85 isomorphism classes of simple cubic graphs on twelve vertices to consider.  Of the 85 graphs, however, 63 have girth three and hence can be ignored by Corollary \ref{cor:girth}.  Amongst the remaining 22 graphs, 20 have girth four.  Eighteen of these twenty graphs have at least two disjoint cycles of length four, making it easy to see that these graphs can be ignored.  The two remaining girth four graphs have disjoint cycles of length four and five, similarly making it easy to see that these graphs can be ignored.  Finally, the remaining two girth five graphs must be considered directly.

\begin{figure}[htpb]
\centering
\includegraphics[height=3.75 cm]{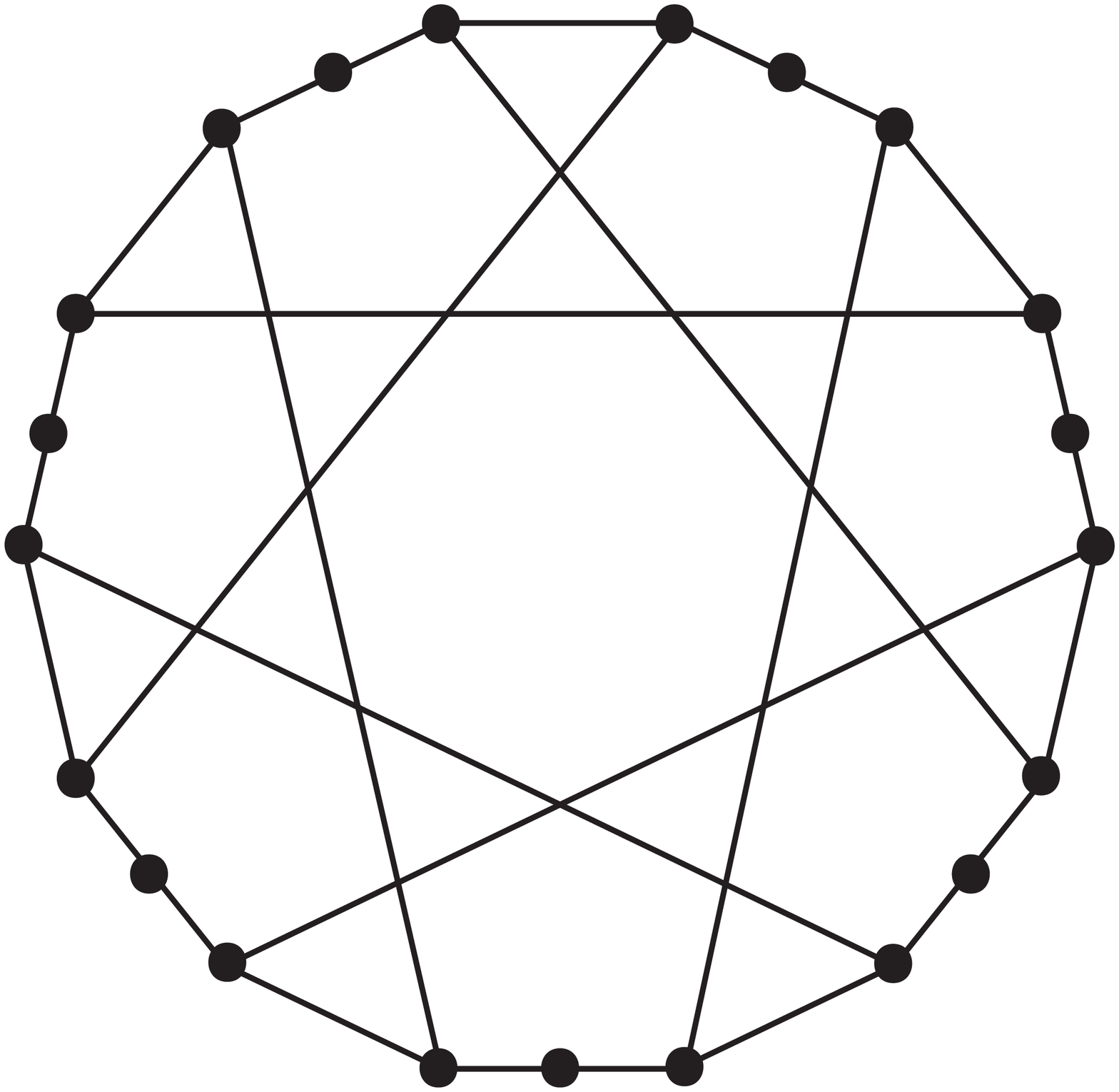}
\caption{A pants decomposition graph $\Gamma(P)$ corresponding to a pants decomposition $P \in \mathcal{P}(S_{8,7})$, obtained from adding boundary components to the Heawood graph.  The pants decomposition $P$ is distance six from any pants decomposition containing a separating curve.  } \label{fig:heawood}
\end{figure}

The genus eight column of the table arises from considering the so called Heawood graph, or the $(3,6)$-cage.  By Note \ref{note:s}, in total there are 509 isomorphism classes of simple cubic graphs on fourteen vertices to consider.  However, by Corollary \ref{cor:girth}, the graphs we only need to consider those that have girth at least five.  There are only nine such graphs to consider.  Of the nine graphs, the eight graphs excluding the Heawood graph have disjoint cycles of length four, making it easy to see that these graphs cannot produce examples of graphs that are further from a pants decomposition containing a separating curve than graphs that can be produced by adding boundary components to the Heawood graph.  See Figure \ref{fig:heawood} for a pants decomposition graph corresponding to a pants decomposition of $S_{8,7}$ which is distance six from a pants decomposition containing a separating curve.
 
% \bibliographystyle{amsalpha}
%\bibliography{bib.bib}

\end{document}